\documentclass[UTF8,a4paper]{article}
\usepackage{amsmath,amsthm,amsfonts,amssymb,amscd}
\usepackage{fix-cm}
\usepackage{latexsym}
\usepackage{caption2}     
\usepackage{float}        
\usepackage{fancyhdr}     
\usepackage{graphicx}     
\usepackage{lastpage}     
\graphicspath{{figures/}} 
\usepackage[top=3cm,bottom=3cm,left=3cm,right=3cm]{geometry}  
\usepackage{fancybox} 
\usepackage{mathrsfs}

\usepackage{xcolor}
\usepackage[colorlinks=true,
linkcolor=blue,citecolor=blue,
urlcolor=blue]{hyperref}

\usepackage{indentfirst}

\theoremstyle{plain}
\newtheorem{thm}{Theorem\hspace{0.05pt}}[section]
\newtheorem{cor}[thm]{Corollary\hspace{0.05pt}}
\newtheorem{lem}[thm]{Lemma\hspace{0.05pt}}

\newtheorem{con}[thm]{Conjecture\hspace{0.05pt}}

\theoremstyle{definition}
\newtheorem{definition}[thm]{Definition\hspace{0.05pt}}

\newtheorem{remark}{Remark\hspace{0.05pt}}

\begin{document}
\title{Extremal spectral radius of degree-based weighted adjacency matrices of graphs with given order and size\footnote{\noindent This work is    partially supported by 
			the National Natural Science Foundation of China (No.~12271182).\\ 
   E-mail: 2011283@tongji.edu.cn (C. Shen),shan\_haiying@tongji.edu.cn (H. Shan)}}

\author{ Chenghao Shen,\, Haiying Shan\thanks{Corresponding author}
\\ \small School of Mathematical Sciences, Key Laboratory of Intelligent Computing and Applications\\ \small(Ministry of Education), Tongji University, Shanghai, China}

\date{}
\maketitle

\begin{abstract}
The $f$ adjacency matrix is a type of edge-weighted adjacency matrix, whose weight of an edge $ij$ is $f(d_i,d_j)$, where $f$ is a real symmetric function and $d_i,d_j$ are the degrees of vertex $i$ and vertex $j$. The $f$-spectral radius of a graph is the spectral radius of its $f$-adjacency matrix. In this paper, the effect of subdividing an edge on $f$-spectral radius is discussed. Some necessary conditions of the extremal graph with given order and size are derived. As an example, we obtain the bicyclic graph(s) with the smallest $f$-spectral radius for fixed order $n\geq8$ by applying generalized Lu-Man method. 
\end{abstract}


\section{Introduction}
\par In this paper, we only consider simple graphs. For basic graph and graph spectra notations, we refer to \cite{bondy2008graph} and \cite{cvetkovic2010introduction}.  We often write $d_{v_i}$ or $d_i$ to refer $d_G(v_i)$, the degree of vertex $v_i$ in graph $G$, when there is no ambiguity.

\par Das et al. \cite{das2018degree} first formally gave the following definition of the weighted adjacency matrix weighted by degree-based indice $f$ ($f$-adjacency matrix for short). 
\begin{definition}
Let $G=(V,E)$ be a connected graph, $V=\{v_i|i\in[n]\}$, $d_i=d(v_i)$ is the degree of $v_i$. Let $f(x,y)>0$ $(x,y>0)$ be a symmetric real function. Denote $A_f(G)$ be the $f$-adjacency matrix of $G$, whose $ij$-entry is defined as
$$(A_f(G))_{i,j}=\begin{cases}
f(d_i,d_j), & \text{ if } v_iv_j\in E(G), \\ 
0, & \text{ otherwise. }
\end{cases}$$
Denote $\rho_f(G):=\rho(A_f(G))$ be the $f$-spectral radius of $G$. 
\end{definition}

Matrices that weighted by specified indices have been studied separately, including Atom-Bond Connectivity matrix (where $f(x,y)=\sqrt{(x+y-2)/(xy)}$, also known as ABC matrix) \cite{estrada1998atom, chen2019extremality}, Randi\'c matrix (where $f(x,y)=1/\sqrt{xy}$)\cite{randic1975characterization, liu2012note}, Sombor matrix (where $f(x,y)=\sqrt{x^2+y^2}$)\cite{Gutman2021SpectrumAE}, etc. 
\par Li et al. \cite{li2021trees} first proposed the idea to unify the spectral study of this type of matrices. They also obtained the extremal tree when adding some particular restrictions on $f(x,y)$. 

\begin{thm}[\cite{li2021trees}]
Assume that $f(x,y)>0$ is a symmetric real function, increasing and convex in variable $x$. Then the extremal tree with the largest $\rho_f$ in $\mathcal{T}_n$ is a star or a double star. 
\end{thm}

\begin{thm}[\cite{li2021trees}]
Assume that $f(x,y)>0$ has a form $P(x,y)$ or $\sqrt{P(x,y)}$, where $P(x,y)$ is a symmetric polynomial with nonnegative coefficients and zero constant term. Then the extremal tree with the smallest $\rho_f$ in $\mathcal{T}_n$ ($n\geq9$) is the path $P_n$. 
\end{thm}

\par If $f(x,y)>0$ is a symmetric real function, increasing and convex in variable $x$ and for any $x_1+y_1=x_2+y_2$ and $|x_1-y_1|>|x_2-y_2|$, $f(x_1,y_1)\geq f(x_2,y_2)$, then $f$ is called to have property $P^*$, which is a further restriction on the function $f(x,y)$. Zheng et al. \cite{zheng2023extremal} obtained the extremal unicyclic graph when $f(x,y)$ has the property $P^*$. 

\begin{thm}[\cite{zheng2023extremal}]
Assume that $f(x,y)$ has the property $P^*$. Among all unicyclic graphs of order $n\geq7$, $C_n$ is the unique unicyclic graph with the smallest $\rho_f$, $S_n+e$ is the unique unicyclic graph with the largest $\rho_f$. 
\end{thm}

\par Let $G_1$ be the graph obtained from the unique bicyclic graph of order 5 by joining $n-4$ isolated vertices to a vertex of degree 3. Ye et al. \cite{ye2023extremal} obtained that $G_1$ is the extremal bicyclic graph of order $n$ with the smallest $\rho_f$ when $f(x,y)$ has the property $P^*$. 

\begin{thm}[\cite{ye2023extremal}]
Let $G$ be a graph in $\mathcal{B}_n$, $n\geq4$. If $f(x,y)$ has the property $P^*$, then $\rho_f(G)\leq\rho_f(G_1)$, with equality holds if and only if $G\cong G_1$. 
\end{thm}

\par Given an edge $e$ of $G$, denote $G_e$ be the graph obtained by subdividing $e$ on $G$. Li et al. \cite{li2023some} obtained uniform interlacing inequalities under edge subdivision operation. 
\begin{thm}[\cite{li2023some}]\label{sdv}
Let $G$ be a graph of order $n$ and $H=G_e$, where $e=uv$ is an edge of $G$. Let $f(x,y)$ be any symmetric real function and the edge-weight $f(d_i,d_j)\geq0$ for any edge $v_iv_j\in E(G)$. If $\lambda_1\geq\lambda_2\geq\dots\geq\lambda_n$ and $\theta_1\geq\theta_2\geq\dots\geq\theta_{n+1}$ are the eigenvalues of $A_f(G)$ and $A_f(H)$, respectively, then $$\lambda_{i-2}\geq\theta_i\geq\lambda_{i+1}$$ for each $i=1,2,\dots,n+1$, where $\lambda_i=+\infty$ for each $i\leq0$ and $\lambda_i=-\infty$ for each $i\geq n-1$. 
\end{thm}
\par The remainder of this paper is organized as follows: In Section 2, the generalized `Lu-Man' method and other basic results are listed. In Section 3, the effect of subdividing a specific edge on $f$-spectral radius is discussed. In Sections 4 and 5, we derive some necessary conditions for the extremal graph with the smallest and largest $\rho_f$, respectively. The bicyclic graphs with the smallest $\rho_f$ are also determined in Section 4.

\section{Preliminaries}

Let $G$ be a connected graph and $H$ be a proper subgraph of $G$. If $f(x,y)>0$ is a symmetric real function, increasing in variable $x$, then $\rho_f(H)<\rho_f(G)$ holds by the following lemma. 
\begin{lem}[\cite{brouwer2012spectra}]\label{subg}
Let $A$, $B$ be nonnegative matrices and $A<B$. If $B$ is irreducible, then $\rho(A)<\rho(B)$. 
\end{lem}

\par Wang et al. \cite{wang2023maximum} generalized the `Lu-Man' method developed in \cite{lu2016connected} for hypergraphs to weighted $k$-uniform hypergraphs, which can be applied to our case. 

\begin{definition}[\cite{wang2023maximum}]
Let $H$ be a connected weighted $k$-uniform hypergraph. If there exists a weighted incidence matrix $B$ satisfying: 
\begin{itemize}
\item[(1)] $\sum_{e\in E(v)}B(v,e)=1$ for any $v\in V(H)$ and $\prod_{v\in e}\frac{B(v,e)}{w_H(e)}=\alpha$ for any $e\in E(H)$ , then $H$ is called $\alpha$-normal; 
\item[(2)] $\sum_{e\in E(v)}B(v,e)\leq1$ for any $v\in V(H)$ and $\prod_{v\in e}\frac{B(v,e)}{w_H(e)}\geq\alpha$ for any $e\in E(H)$ , then $H$ is called $\alpha$-subnormal; $H$ is called strictly $\alpha$-subnormal if it is $\alpha$-subnormal but not $\alpha$-normal; 
\item[(3)] $\sum_{e\in E(v)}B(v,e)\geq1$ for any $v\in V(H)$ and $\prod_{v\in e}\frac{B(v,e)}{w_H(e)}\leq\alpha$ for any $e\in E(H)$ , then $H$ is called $\alpha$-supernormal; $H$ is called strictly $\alpha$-supernormal if it is $\alpha$-supernormal but not $\alpha$-normal. 
\end{itemize}
\par Moreover, the incidence matrix $B$ is called consistent if $\prod_{i=1}^l\frac{B(v_i,e_i)}{B(v_{i-1},e_i)}=1$ holds for any cycle $v_0e_1v_1\dots v_{l-1}e_lv_l$ ($v_l=v_0$) of $H$. In this case, we call $H$ consistently $\alpha$-normal (or correspondingly, consistently $\alpha$-supernormal, consistently and strictly $\alpha$-supernormal). 
\end{definition}

\begin{lem}[\cite{wang2023maximum}]\label{lmm}
Let $H$ be a connected weighted $k$-uniform hypergraph, where $k\geq2$ and order $n\geq3$. We have: 
\begin{itemize}
\item[(1)] $\rho(H)=\alpha^{-\frac1k}$ if and only if $H$ is consistently $\alpha$-normal; 
\item[(2)] If $H$ is $\alpha$-subnormal, then $\rho(H)\leq\alpha^{-\frac1k}$; 
\item[(3)] If $H$ is strictly $\alpha$-subnormal, then $\rho(H)<\alpha^{-\frac1k}$; 
\item[(4)] If $H$ is consistently $\alpha$-supernormal, then $\rho(H)\geq\alpha^{-\frac1k}$; 
\item[(5)] If $H$ is consistently and strictly $\alpha$-supernormal, then $\rho(H)>\alpha^{-\frac1k}$. 
\end{itemize}
\end{lem}

From the proof of the above lemma from \cite{wang2023maximum}, we can define the principal weighted incidence matrix $B_H$ of $H$ as follows: 
$$B_H(v,e)=\begin{cases}
\frac{w_H(e)\boldsymbol{x}^e}{\rho x_v^k}, & \quad \text{if }v\in e, \\ 
0, & \quad \text{otherwise}, 
\end{cases}$$
where $\boldsymbol{x}$ is the principal eigenvector of $H$, $\boldsymbol{x}^e=\prod_{v\in e}x_v$. Denote $\alpha(H):=\rho^{-k}(H)$, so $H$ is consistently $\alpha(H)$-normal. 

\section{Subdividing an edge}
\par In this section, $f(x,y)>0$ is a symmetric real function, increasing in variable $x$ and we focus on the effect of subdividing an edge on the $f$-spectral radius. 
\par For an arbitrary edge $e$, we can only obtain $\rho_f(G_e)\geq\lambda_2(A_f(G))$ and $\rho_f(G)\geq\lambda_3(A_f(G_e))$ from Theorem \ref{sdv}. A specific edge $e$ is required to compare $\rho_f(G_e)$ and $\rho_f(G)$. 

\begin{lem}\label{lem1}
Assume that $f(x,y)>0$ is a symmetric real function, increasing in variable $x$. Given a connected graph $G$ and a vertex $v$ with $d_v\geq2$, denote $u_1,u_2$ be two neighbors of $v$. Let $\boldsymbol{x}$ be the principal eigenvector of $A_f(G)$ and $\boldsymbol{y}$ be a vector such that $y_i=f(d_i,2)x_i$ for every vertex $i$. If $d_{u_i}\geq 2$ and $y_v\leq y_{u_i}$ $(i=1,2)$, then $\rho_f(G_{vu_1})\leq\rho_f(G)$. 
\end{lem}
\begin{proof}
Denote $\rho=\rho_f(G)$, $G'=G_{vu_1}$. Let $w$ be the new vertex when subdividing $vu_1$ on $G$. Let $\boldsymbol{x'}$ be a vector on $G'$, where
$$x_i':=\begin{cases}
x_v, & \quad i=w, \\ 
x_i, & \quad i\neq w.
\end{cases}$$
\par Since 
$$f(d_{u_1},d_v)x_{u_1}\geq f(d_{u_1},2)x_{u_1}=y_{u_1}\geq y_v=f(2,d_v)x_w', $$
$$f(d_{u_2},d_v)x_{u_2}\geq f(d_{u_2},2)x_{u_2}=y_{u_2}\geq y_v=f(d_v,2)x_v', $$
$$f(d_{u_1},d_v)x_v\geq f(d_{u_1},2)x_v=f(d_{u_1},2)x_w', $$
we have 
$$
\begin{aligned}
 \rho x_v'=&\sum_{iv\in E(G)}f(d_i,d_v)x_i=\sum_{iv\in E(G')}f(d_i,d_v)x_i'+(f(d_{u_1},d_v)x_{u_1}-f(2,d_v)x_w')\\
 \geq&\sum_{iv\in E(G')}f(d_i,d_v)x_i',    
\end{aligned}
$$
$$\begin{aligned}
\rho x_w'=&\rho x_v\geq f(d_{u_1},d_v)x_{u_1}+f(d_{u_2},d_v)x_{u_2}\geq f(d_{u_1},2)x_{u_1}'+f(d_v,2)x_v', \\
\rho x_{u_1}'=&\sum_{iu_1\in E(G)}f(d_i,d_{u_1})x_i=\sum_{iu_1\in E(G')}f(d_i,d_{u_1})x_i'+(f(d_{u_1},d_v)x_v-f(d_{u_1},2)x_w')\\
\geq&\sum_{iu_1\in E(G')}f(d_i,d_{u_1})x_{u_1}', 
\end{aligned}$$
$$\rho x_j'=\sum_{ij\in E(G')}f(d_i,d_j)x_j', \qquad j\neq v,w,u_1. $$
\par So $\rho\boldsymbol{x'}\geq A_f(G')\boldsymbol{x'}$. By Perron-Frobenius Theorem, $\rho_f(G')\leq\rho$. 
\end{proof}

\begin{thm}\label{thm1}
Assume that $f(x,y)>0$ is a symmetric real function, increasing in variable $x$. Given a connected graph $G$, $C$ is a cycle of $G$. Then there exists an edge $e$ on $C$, such that $\rho_f(G_e)\leq\rho_f(G)$. 
\end{thm}
\begin{proof}
Let $C=v_1v_2\dots v_lv_1$. Without loss of generality, assume that $y_1=\min\{y_1,y_2,\dots,y_l\}$, so $y_1\leq y_2$ and $y_1\leq y_l$. By Lemma \ref{lem1}, $\rho_f(G_{v_1v_2})\leq\rho_f(G)$. 
\end{proof}

\begin{thm}\label{thm2}
Assume that $f(x,y)>0$ is a symmetric real function, increasing in variable $x$. Given a connected graph $G$, $v_1$ is a 2-degree vertex of $G$. Denote its neighbors by $v_2$ and $v_3$. If $v_2v_3\in E(G)$, then $\rho_f(G_{v_1v_2})\leq\rho_f(G)$. 
\end{thm}
\begin{proof}
By Lemma \ref{lem1}, it suffices to show that $x_1\leq x_2$, then $y_1\leq y_2$ (and similarly, $y_1\leq y_3$). Assume that $x_1>x_2$. Denote $\rho=\rho(G)$, then 
$$\rho x_2<\rho x_1=f(d_2,2)x_2+f(d_3,2)x_3, $$
while 
$$\rho x_2\geq f(2,d_2)x_1+f(d_3,d_2)x_3>f(2,d_2)x_2+f(d_3,2)x_3>\rho x_2, $$
a contradiction. 
\end{proof}

\section{Applying generalized Lu-Man method on an internal path}
An internal path is a sequence of vertices $v_0,\dots,v_l$ such that $d_0\geq3$, $d_l\geq3$, $d_1=\dots=d_{l-1}=2$ and $v_i\sim v_{i+1}$ for $i=0,1,\dots,l-1$. Note that an internal path could be a cycle when $v_0=v_l$. 

\par Let $G$ be a connected graph and $P=v_0e_1v_1e_2\dots v_{l-1}e_lv_l$ be an internal path in $G$, $l\geq2$. If there exists a weighted incidence matrix $B$, such that
$$\sum_{e\in E(v)}B(v,e)=1 \text{ for any }v\in \{v_1,\dots,v_{l-1}\}, $$
$$\prod_{v\in e}B(v,e)=\alpha w(e)^2 \text{ for any }e\in \{e_1,\dots,e_l\}, $$
then $P$ is called $(B(v_0,e_1),B(v_l,e_l))$ $\alpha$-normal (with $B$). Denote $\alpha':=(f(2,2))^2 \alpha $, $\beta(d):=(f(d,2)/f(2,2))^2$, for $d_0=d_G(v_0)$, we have
$$B(v_{i-1},e_i)(1-B(v_i,e_{i+1}))=\begin{cases}
\alpha', & \quad i=2,3,\dots,l-1, \\ 
\alpha'\beta(d_0), & \quad i=1. 
\end{cases}$$
Let
$$x_i:=\begin{cases}
B(v_i,e_{i+1}), & \quad i=1,2,\dots,l-1, \\ 
\beta(d_0)^{-1}B(v_i,e_{i+1}), & \quad i=0, 
\end{cases}$$
then $(x_i)_{i=0}^{l-1}$ is a subsequence of the following bi-infinite sequence $(x_i)_{i\in\mathbb{Z}}$ such that
$$x_n=f(x_{n-1})=\dots=f^n(x_0)\quad\text{and}\quad x_{-n}=f^{-1}(x_{1-n})=\dots=f^{-n}(x_0)\quad\text{for } n\in\mathbb{N}, $$
where 
$$f(x)=1-\frac{\alpha'}x. $$
Shan et al. \cite{shan2021smallest} discussed sequences of this type. 

\begin{lem}[\cite{shan2021smallest}]\label{lem2}
Let $(x_i)_{i\in\mathbb{Z}}$ be a bi-infinite sequence such that $x_n=f(x_{n-1})$, where $f(x)=1-\frac{\alpha'}x$. Suppose $x_p+x_q=1$ ($p,q\in\mathbb{Z}$). If $\alpha'\in(0,\frac14]$, then $x_n$ can be expressed as:
$$x_n=F_\theta(p+q-2n), $$
where
$$F_\theta(x):=\frac12\left(1-\tanh(\theta)\tanh\left(\frac{x\theta}2\right)\right), $$
$$\theta:=\mathrm{arccosh}\left(\frac12\alpha'^{-\frac12}\right). $$
And $F_\theta(x)$ is strictly decreasing and convex in $[0,+\infty)$. 
\end{lem}

\begin{remark}
For any graph $G$ that is not a tree, there exists a cycle $C_g$ as a subgraph, so $\rho_f(G)\geq\rho_f(C_g)=2f(2,2)$ and $\alpha'(G)=(\rho(f(G)))^{-2}(f(2,2))^2\in(0,\frac14]$ hold. 
\end{remark}

Let $u,v$ be two vertices of graph $G$.
If there exists an automorphism  of $G$ which maps a vertex $u$ to another vertex $v$, then $u$ and $v$ are referred to as similar vertices.

The following lemma shows the application of Lemma \ref{lem2}, which plays an important role in Section 5. 
\begin{lem}\label{lem4}
Let $G$ be a connected graph and $P=v_0e_1v_1e_2\dots v_{l-1}e_lv_l$ be an internal path in $G$, $l\geq2$. Let $F_\theta(x)$ be defined as above. Setting $c_1=\beta(d_G(v_0))F_\theta(l)),c_2=\beta(d_G(v_0))F_\theta(l_1)$ and $c_3=\beta(d_G(v_l))F_\theta(l_2)$, then
\begin{itemize}
\item[(1)] For $\alpha=\alpha(G)\in(0,\frac14(f(2,2))^{-2}]$, if $v_0$ and $v_l$ are similar vertices, then $P$ is $(c_1,c_1)$ $\alpha$-normal with the principal weighted incidence matrix $B_G$. 
\item[(2)] For any $\alpha\in(0,\frac14f(f(2,2))^{-2}]$ and integer $l_1,l_2$ with $l_1+l_2=2l$, $P$ is $(c_2,c_3)$ $\alpha$-normal. 
\end{itemize}
\end{lem}
\begin{proof}
\begin{itemize}
\item[(1)] Since $v_0$ and $v_l$ are similar vertices, $B_G(v_1,e_2)=B_G(v_{l-1},e_{l-1})$ by symmetry. 
So $x_1+x_{l-1}=B_G(v_1,e_2)+B_G(v_{l-1},e_{l})=B_G(v_{l-1},e_{l-1})+B_G(v_{l-1},e_{l})=1$. By Lemma \ref{lem2}, $B_G(v_l,e_l)=B_G(v_0,e_1)=\beta(d_0)x_0=\beta(d_0)F_\theta(l)$. 
\item[(2)] 
Take $B(v_0,e_1)=\beta(d_0)F_\theta(l_1)$, then $x_0=\beta(d_0)^{-1}B(v_0,e_1)=F_\theta(l_1)$, and $x_n=f^n(x_0)=F_\theta(l_1-2n)$. So $x_0+x_{l_1}=1$. Let $y_{l-i}=1-x_i$, 
then $y_l+y_{l-l_1}=(1-x_0)+(1-x_{l_1})=1$. By Lemma \ref{lem2}, $y_0=F_\theta(l+(l-l_1)-2\cdot0)=F_\theta(l_2)$, so $B(v_l,e_l)=\beta(d_l)y_0=\beta(d_l)F_\theta(l_2)$. 
\end{itemize}
\end{proof}

\begin{remark}
The result does not hold directly for $l=1$, since the weight of $e_1=v_0v_l$ is $f(d_0,d_l)$, while the deduction requires it to be $f(d_0,2)$. However, we can manually modify the weight of $e_1$ from $f(d_0,d_l)$ to $f(d_0,2)$ to keep the deduction. Note that this modification decreases the spectral radius, so we need to be careful when applying Lemma \ref{lem4} for $l=1$. 
\end{remark}

And we need two more propositions on $F_\theta(x)$. 

\begin{lem}\label{lem3}
Let $F_\theta(x)$ be defined as above, $a\geq b\geq0$, then
$$F_\theta(a)+F_\theta(-b)\geq F_\theta(a-b)+F_\theta(0). $$
\end{lem}
\begin{proof}
Since $F_\theta(x)$ is convex in $[0,+\infty)$ and $0\leq a-b\leq a$, $0\leq b\leq a$, we have
$$F_\theta(a-b)+F_\theta(b)\leq F_\theta(0)+F_\theta(a). $$
Note that
$$F_\theta(0)=\frac12,\quad F_\theta(b)+F_\theta(-b)=1, $$
so
$$F_\theta(a-b)+F_\theta(0)-F_\theta(-b)=F_\theta(a-b)+F_\theta(b)-F_\theta(0)\leq F_\theta(a). $$
\end{proof}

\begin{lem}\label{lem44}
Let $F_\theta(x)$ be defined as above, $x\geq3$, $\theta>0$, then
$$4F_\theta(2x)>3F_\theta(x). $$
\end{lem}
\begin{proof}
Let $f_1(x):=8F_\theta(4x)-6F_\theta(2x)=1+3\tanh(\theta)\tanh(\theta x)-4\tanh(\theta)\tanh(2\theta x)$, $x\geq\frac32$, 
$$f_1'(x)=\theta \tanh(\theta)(8\tanh^2(2\theta x)-3\tanh^2(\theta x)-5). $$
The solution of $f_1'(x_0)=0$ on $x_0>0$ is $x_0=\tau/\theta$, where $\tau=\mathrm{arctanh}(1/\sqrt{3})\approx0.6585$. So $f_1(x)$ decreases in $(0,x_0]$, increases in $[x_0,+\infty)$. 
\par If $\theta\leq\frac23\tau<\tau$, then $$
\begin{aligned}
f_1(x)\geq &f_1(\tau/\theta)=1-(4\tanh(2\tau)-3\tanh(\tau))\tanh(\theta)\\
=&1-(3-\sqrt{3})\tanh(\theta)>1-(3-\sqrt{3})\tanh(\tau)=2-\sqrt{3}>0.  
\end{aligned}
$$ 
\par If $\theta>\frac23\tau$, let $t=e^\theta$, then
\begin{align*}
f_1(x)\geq f_1\left(\frac32\right)&=1+3\tanh(\theta)\tanh\left(\frac32\theta\right)-4\tanh(\theta)\tanh(3\theta)\\
&=\frac{2\left((t-2)^2t^4+(2t-1)^2\right)t^2}{(t^4-t^2+1)(t^2-t+1)(t^2+1)^2}\\
&>0.\qquad (\text{Since } t>1)
\end{align*}
\par So $4F_\theta(2x)>3F_\theta(x)$ when $x\geq3$. 
\end{proof}

\section{The $c$-cyclic graph with the smallest $\rho_f$}
\par Let $\mathcal{G}_{c,n}$ be the set of all connected graphs with cyclomatic number $c$ and order $n$. In particular, $\mathcal{G}_{0,n}$, $\mathcal{G}_{1,n}$ and $\mathcal{G}_{2,n}$ represent the set of all trees, unicyclic graphs and bicyclic graphs of order $n$ and are often written as $\mathcal{T}_n$, $\mathcal{U}_n$ and $\mathcal{B}_n$, respectively. The base of a $c$-cyclic graph $G$, denoted by $\widehat{G}$, is the (unique) minimal $c$-cyclic subgraph of $G$. It is easy to see that $\widehat{G}$ can be obtained from $G$ by consecutively deleting pendent vertices. 
\par In this section, we will discuss the extremal graph of cyclomatic number $c$ and order $n$ with the smallest $\rho_f$. We assume that $f(x,y)>0$ is a symmetric real function, increasing in variable $x$. 


\begin{thm}\label{thm51}
Let $G$ be the graph with the smallest $\rho_f$ in $\mathcal{G}_{c,n}$, $c\geq1$, then $G=\widehat{G}$. 
\end{thm}
\begin{proof}
Otherwise, let $v$ be a pendant vertex and $G':=G-v$, so $\rho_f(G')<\rho_f(G)$. By Theorem \ref{thm1}, there exists an edge $e \in E(\widehat{G}')$, such that $\rho_f(G'_e)\leq\rho_f(G')$. So $\rho_f(G'_e)<\rho_f(G)$ and $G'_e\in\mathcal{G}_{c,n}$, a contradiction. 
\end{proof}
It's enough to show that $C_n$ is the extremal graph with the smallest $\rho_f(G)$ in unicyclic graphs with order $n$. For bicyclic graphs, we need more discussions. 
\par It's known that there are three types of bicyclic graphs without pendant vertices: 
\begin{itemize}
\item[(1)] A $\Theta$-type graph $\Theta(l_1,l_2,l_3)$ is obtained from three pairwise disjoint paths of length $l_1,l_2,l_3$ from vertex $u$ to $v$; 
\item[(2)] An $\infty$-type graph $\infty(l_1,l_2,l_3)$ is obtained from two vertex-disjoint cycles $C_{l_1}$ and $C_{l_2}$ by joining vertex $u$ of $C_{l_1}$ and vertex $v$ of $C_{l_2}$ by a path of length $l_3\geq1$; 
\item[(3)] An $\infty^*$-type graph $\infty^*(l_1,l_2)$ is obtained from two vertex-disjoint cycles $C_{l_1}$ and $C_{l_2}$ by identifying vertex $u$ of $C_{l_1}$ and vertex $v$ of $C_{l_2}$. 
\end{itemize}

\begin{thm}\label{thm52}
Let $s,t$ be positive integers such that $2s+t=m\geq6$, $|s-t|\leq1$. Then $\Theta(s,s,t)$ is the extremal graph with the smallest $\rho_f$ in all $\Theta$-type graphs with size $m$. 
\end{thm}
\begin{proof}
Let $G=\Theta(l_1,l_2,l_3)\ncong\Theta(s,s,t)$, $l_1\leq l_2\leq l_3$. By Lemma \ref{lem2}, $F_\theta(x)$ is strictly convex in $[0,+\infty)$. Then 
$$F_\theta(l_1)+F_\theta(l_2)+F_\theta(l_3)>F_\theta(s)+F_\theta(s)+F_\theta(t).$$
 Take $\alpha=\alpha(G)$ and $u$ as a vertex of degree 3 of $G$.  By Lemma \ref{lem4}(1), We have $$\sum_{e\in E_G(u)}B_G(u,e)=\beta(3)(F_\theta(l_1)+F_\theta(l_2)+F_\theta(l_3))=1.$$ (If $l_1=1$, we need to modify the weight of edge $uv$ from $f(3,3)$ to $f(3,2)$ to keep the equation; The result still holds since the spectral radius decreases after the modification and we want to prove that $\rho_f(G)>\rho_f(G')$. )
\par For $G'=\Theta(s,s,t)$, denote $u',v'$ as vertices of degree 3 of $G'$. Since $s,t\geq2$, by Lemma \ref{lem4}(2), there exists a weighted incidence matrix $B$, such that $B$ is $(\beta(3)F_\theta(s),\beta(3)F_\theta(s))$, $(\beta(3)F_\theta(s),\beta(3)F_\theta(s))$, $(\beta(3)F_\theta(t),\beta(3)F_\theta(t))$ $\alpha$-normal on the three $u'$-$v'$ paths respectively. And
\begin{align*}
\sum_{e\in E_{G'}(u')}B(u',e)=\sum_{e\in E_{G'}(v')}B(v',e)&=\beta(3)(F_\theta(s)+F_\theta(s)+F_\theta(t))\\
&<\beta(3)(F_\theta(l_1)+F_\theta(l_2)+F_\theta(l_3))=1, 
\end{align*}
so $G'$ is strictly $\alpha$-subnormal. By Lemma \ref{lmm}(3), the result is derived. 
\end{proof}

\begin{thm}\label{thm53}
Let $s,t$ be positive integers such that $2s+t=m\geq8$, $|s-t|\leq1$. Then $\infty(s,s,t)$ is the extremal graph with the smallest $\rho_f$ in all $\infty$-type graphs with size $m$. 
\end{thm}
\begin{proof}
Let $G=\infty(l_1,l_2,l_3)\ncong\infty(s,s,t)=G'$. If $m-2l_i\geq0$, by Lemma \ref{lem2}, 
$$2F_\theta(l_i)+F_\theta(m-2l_i)\geq2F_\theta(s)+F_\theta(t), $$
with equality hold if and only if $l_i=s$. 
\par If $m-2l_i<0$, note that $l_i>2l_i-m$. By Lemma \ref{lem3} and Lemma \ref{lem2}, 
$$2F_\theta(l_i)+F_\theta(m-2l_i)\geq F_\theta(l_i)+F_\theta(m-l_i)+F_\theta(0)>2F_\theta(s)+F_\theta(t). $$
\par So $2F_\theta(l_i)+F_\theta(m-2l_i)\geq2F_\theta(s)+F_\theta(t)$ ($i=1,2$), with both equalities cannot hold at the same time. 
Take $\alpha=\alpha(G')$ and $u'$ as a vertex of degree 3 of $G'$, so $\sum_{e\in E_{G'}(u')}B_{G'}(u',e)=\beta(3)(2F_\theta(s)+F_\theta(m-2s))=1$ by Lemma \ref{lem4}(1). For $G$,  denote $u,v$ as vertices of degree 3 of $G$. By Lemma \ref{lem4}(2), there exists a weighted incidence matrix $B$, such that $B$ is $(\beta(3)F_\theta(l_1),\beta(3)F_\theta(l_1))$, $(\beta(3)F_\theta(l_2),\beta(3)F_\theta(l_2))$, $(\beta(3)F_\theta(m-2l_1),\beta(3)F_\theta(m-2l_2))$ $\alpha$-normal on the two cycles and the $u$-$v$ path respectively (If $l_3=1$, we need to modify the weight of edge $uv$ from $f(3,3)$ to $f(3,2)$ similarly to the proof of Theorem \ref{thm52}). $B$ is consistent with the two cycles by symmetry. And 
$$\sum_{e\in E_G(u)}B(u,e)=\beta(3)(2F_\theta(l_1)+F_\theta(m-2l_1))\geq\beta(3)(2F_\theta(s)+F_\theta(m-2s))=1,$$
$$\sum_{e\in E_G(v)}B(v,e)=\beta(3)(2F_\theta(l_2)+F_\theta(m-2l_2))\geq\beta(3)(2F_\theta(s)+F_\theta(m-2s))=1,$$
with both equalities cannot hold at the same time. So $G$ is consistently and strictly $\alpha$-supernormal. By Lemma \ref{lmm}(5), the result is derived. 
\end{proof}

\begin{thm}\label{thm54}
Let $s\geq3,t\geq2$ be integers, then $\rho_f(\Theta(s,s,t))=\rho_f(\infty(s,s,t))$. 
\end{thm}

\begin{proof}
Denote $G=\Theta(s,s,t)$, $\alpha=\alpha(G)$ and $u$ as a vertex of degree 3 of $G'$, so $\sum_{e\in E_{G}(u)}B_{G}(u,e)=\beta(3)(F_\theta(s)+F_\theta(s)+F_\theta(t))=1$ by Lemma \ref{lem4}(1). For $G'=\infty(s,s,t)$, denote $u',v'$ as vertices of degree 3 of $G'$. By Lemma \ref{lem4}(2), There exists a weighted incidence matrix $B$, such that $B$ is $(\beta(3)F_\theta(s),\beta(3)F_\theta(s))$, $(\beta(3)F_\theta(s),\beta(3)F_\theta(s))$, $(\beta(3)F_\theta(t),\beta(3)F_\theta(t))$ $\alpha$-normal on the two cycles and the $u'$-$v'$ path respectively. $B$ is consistent with the two cycles by symmetry. And $\sum_{e\in E_{G'}(u')}B(u',e)=\sum_{e\in E_{G'}(v')}B(v',e)=\beta(3)(F_\theta(s)+F_\theta(s)+F_\theta(t))=1$, so $\infty(s,s,t)$ is consistently $\alpha$-normal. By Lemma \ref{lmm}(1), the result is derived. 
\end{proof}

\begin{thm}\label{thm55}
For any $\infty^*$-type graph $\infty^*(l_1,l_2)$ with size $m=l_1+l_2\geq9$, there exists a $\Theta$-type graph $G'$ such that $\rho_f(G')<\rho_f(\infty^*(l_1,l_2))$. 
\end{thm}
\begin{proof}
Denote $G=\infty^*(l_1,l_2)$, $\alpha=\alpha(G)$  and $u$ as the vertex of degree 4 of $G$, then $\sum_{e\in E_{G}(u)}B_{G}(u,e)=\beta(4)(2F_\theta(l_1)+2F_\theta(l_2))=1$ by Lemma \ref{lem4}(1). 
\par Let $s,t$ be positive integers such that $2s+t=m$, $|s-t|\leq1$. Take $G'=\Theta(s,s,t)$ and $u',v'$ as vertices of degree 3 of $G'$. Then there exists a weighted incidence matrix $B$, such that $B$ is $(\beta(3)F_\theta(s),\beta(3)F_\theta(s))$, $(\beta(3)F_\theta(s),\beta(3)F_\theta(s))$, $(\beta(3)F_\theta(t),\beta(3)F_\theta(t))$ $\alpha$-normal on the three $u'$-$v'$ paths respectively by Lemma \ref{lem4}(2). Since $\lfloor\frac m3\rfloor\geq3$, $2\lfloor\frac m3\rfloor\geq\lceil\frac m2\rceil$, by Lemma \ref{lem2} and Lemma \ref{lem44}, we have
\begin{align*}
\sum_{e\in E_{G'}(u')}B(u',e)=\sum_{e\in E_{G'}(v')}B(v',e)&=\beta(3)(2F_\theta(s)+F_\theta(t))\\
&\leq\beta(4)\cdot3F_\theta\left(\left\lfloor\frac m3\right\rfloor\right)\\
&<\beta(4)\cdot4F_\theta\left(2\left\lfloor\frac m3\right\rfloor\right)\\
&\leq\beta(4)\cdot4F_\theta\left(\left\lceil\frac m2\right\rceil\right)\\
&\leq\beta(4)(2F_\theta(l_1)+2F_\theta(l_2))=1. 
\end{align*}
So $G'$ is strictly $\alpha$-subnormal. By Lemma \ref{lmm}(3), the result is derived. 
\end{proof}

\begin{remark}
If $\beta(4)^{-1}+\beta(2)^{-1}\leq2\beta(3)^{-1}$ (it holds for $f\equiv1$, for instance), then the result can be proved directly without using Lemma \ref{lem44}: 
\begin{align*}
\beta(3)(2F_\theta(s)+F_\theta(t))&<\beta(3)(F_\theta(l_1)+F_\theta(l_2)+F_\theta(0))\\&\leq\frac{2}{\beta(4)^{-1}+\beta(2)^{-1}}\left(F_\theta(l_1)+F_\theta(l_2)+F_\theta(0)\right)\\
&=\frac{2\beta(4)}{\beta(4)+1}\left(\frac1{2\beta(4)}+\frac12\right)\\
&=1. 
\end{align*}
But for those $f$ corresponding to common chemical indices (such as the first Zagreb index $f=x+y$, the second Zagreb index $f=xy$, see also Table 1 of \cite{li2021trees}), $\beta(4)^{-1}+\beta(2)^{-1}\leq2\beta(3)^{-1}$ does not hold. 
\end{remark}

From Theorem \ref{thm51}, \ref{thm52}, \ref{thm53}, \ref{thm54}, \ref{thm55}, we derive our main result. 
\begin{thm}
Let $s,t$ be positive integers such that $2s+t-1=n\geq8$, $|s-t|\leq1$. Then $\Theta(s,s,t)$ and $\infty(s,s,t)$ are the extremal graphs with the smallest $\rho_f$ in $\mathcal{B}_n$. 
\end{thm}

\begin{remark}
For $n=5,6,7$, the result does not hold for some $f$. For example, if we take $f(2,2)=1$ and $f(3,2)=f(4,2)=2$, by direct calculation, 
$$\rho_f(\infty^*(3,3))\approx4.5311<4.8990\approx\rho_f(\Theta(2,2,2)), $$
$$\rho_f(\infty^*(4,3))\approx4.3914<4.5949\approx\rho_f(\Theta(3,2,2)), $$
$$\rho_f(\infty^*(4,4))\approx4.2426<4.2930\approx\rho_f(\Theta(3,3,2)). $$
On the contrary, as we expected, 
$$\rho_f(\infty^*(5,4))\approx4.2028>4=\rho_f(\Theta(3,3,3)), $$
$$\rho_f(\infty^*(5,5))\approx4.1613>3.9169\approx\rho_f(\Theta(4,3,3)). $$
\end{remark}

\section{The $c$-cyclic graph with the largest $\rho_f$}
\par In this section, we will discuss the extremal graph of cyclomatic number $c$ and order $n$ with the largest $\rho_f$. We assume that $f(x,y)>0$ is a symmetric real function, increasing and convex in variable $x$. 
\par Li et al. \cite{li2021trees} discussed the effect of the Kelmans operation on $\rho_f$ of trees, which can be directly generalized to common connected graphs. 
\begin{lem}[\cite{li2021trees}]\label{kelman}
Let $G$ be a connected graph and $u, v$ be two vertices of $G$ such that $u\nsim v$. Let $G'$ be the graph obtained by replacing edge $uw$ by a new edge $vw$ for all vertex $w$ such that $u\sim w\nsim v$ (this process can be called as using Kelmans operation on vertex $u$ and $v$ of $G$). If $G'$ is connected and $G\ncong G'$, then $\rho_f(G)<\rho_f(G')$. 
\end{lem}
The authors in \cite{li2021trees,zheng2023extremal,ye2023extremal} obtained their results by consecutively using Kelmans operation and Lemma \ref{kelman}. Or we can have the following short but revealing conclusion to show exactly what the Kelmans operation can do in determining the extremal graph. 
\begin{thm}\label{thm4}
Let $G$ be the graph with the largest $\rho_f$ in $\mathcal{G}_{c,n}$, $c\geq0$, then $P_5$, $C_5$, $C_3\cdot P_3$, $\infty^*(3,3)$, $\Theta(1,2,3)$, $K_5-P_4$ are forbidden from $G$ (which means they cannot exist as an induced subgraph of $G$; $C_3\cdot P_3$ is obtained from $C_3$ by adding a pendant path of length 2 on a vertex of $C_3$). 
\end{thm}
\begin{proof}
Otherwise, there exists $v_1,v_2,v_3,v_4,v_5\in V(G)$, such that $v_1v_2, v_1v_3, v_2v_4, v_3v_5\in E(G)$ and $v_2v_3, v_2v_5, v_3v_4\notin E(G)$. Then we use Kelmans operation on vertex $v_2$ and $v_3$ of $G$ to obtain $G'$, $G'\ncong G$. Since $G'$ is connected, $G'\in\mathcal{G}_{c,n}$. By Lemma \ref{kelman}, $\rho_f(G')>\rho_f(G)$, a contradiction. 
\end{proof}

If $f(x,y)$ has the property $P^*$, then the condition $u\nsim v$ in Lemma \ref{kelman} can be removed. Follow the same approches above, we can obtain that $P_4$ and $C_4$ are also forbidden from the extremal graph $G$ and it's quite easy to determine the extremal tree \cite{zheng2023extremal}, unicyclic \cite{zheng2023extremal} and bicyclic graph \cite{ye2023extremal} with the largest $\rho_f$ in this case. 
\par The following results do not assume that $f(x,y)$ has the property $P^*$ and are obtained directly from Theorem \ref{thm4}. 

\begin{cor}\label{pend}
Let $G$ be the graph with the largest $\rho_f$ in $\mathcal{G}_{c,n}$, $c\geq1$, $v\in V(G)$ and $v\notin V(\widehat{G})$, then $d(v)=1$. 
\end{cor}
\begin{proof}
Otherwise, $P_5$ or $C_3\cdot P_3$ is an induced subgraph of $G$, a contradiction to Theorem \ref{thm4}. 
\end{proof}

\begin{cor}
Let $G$ be the graph with the largest $\rho_f$ in $\mathcal{U}_n$, then $\widehat{G}=C_3$. 
\end{cor}
\begin{proof}
By Theorem \ref{thm4}, $\widehat{G}=C_3$ or $C_4$. If $\widehat{G}=C_4$, by Corollary \ref{pend} and $P_5$ is forbidden, $G$ is the graph obtained from $C_4$ by adding $s$ and $t$ pendant edges to two adjacent vertices of $C_4$, respectively, and can be denoted as $C_4(s,t,0,0)$, where $s+t=n-4$. Similarly, $C_3(s,t,r)$ is defined as the graph obtained by adding the $s,t,r$ pendant edges to the three vertices of $C_3$, respectively. By Theorem \ref{thm2} and Lemma \ref{subg}, 
$$\rho_f(C_4(s,t,0,0))\leq\rho_f(C_3(s,t,0))<\rho_f(C_3(s+1,t,0)), $$
while $C_3(s+1,t,0)\in\mathcal{U}_n$, a contradiction. 
\end{proof}

\begin{cor}
Let $G$ be the graph with the largest $\rho_f$ in $\mathcal{B}_n$, then $\widehat{G}=\Theta(1,2,2)$ or $\Theta(2,2,2)$. 
\end{cor}
\begin{proof}
By Theorem \ref{thm4}, $\widehat{G}$ must be a $\Theta$-type graph, since $\infty^*(3,3)$ or $C_3\cdot P_3$ or $P_5$ exists in an $\infty^*$-type graph and $C_3\cdot P_3$ or $P_5$ exists in an $\infty$-type graph. Let $\widehat{G}=\Theta(r,s,t)$, $r\leq s\leq t$. Note that any induced cycle of $\widehat{G}$ must have length 3 or 4. If $r\geq2$, then $s+t\leq4$, so $\widehat{G}=\Theta(2,2,2)$. If $r=1$, then $r+t\leq4$, so $\widehat{G}=\Theta(1,2,2)$ or $\Theta(1,2,3)$ or $\Theta(1,3,3)$, but $\Theta(1,2,3)$ is forbidden by Theorem \ref{thm4} and $P_5$ exists in $\Theta(1,3,3)$. 
\end{proof}

\par In addition to property $P^*$, we can define property $P^{**}$ as follows. 
\begin{definition}
If $f(x,y)>0$ is a symmetric real function, increasing and convex in variable $x$ and for any $x_1+y_1=x_2+y_2$ and $|x_1-y_1|>|x_2-y_2|$, $f(x_1,y_1)<f(x_2,y_2)$, then $f$ is called to have property $P^{**}$. 
\end{definition}

Functions of this type contain some common chemical indices such as the second Zagreb index $f=xy$ and Reciprocal Randi\'c index $f=\sqrt{xy}$. These functions show quite different behaviors in contrast to $P^*$-type functions in direct calculations on the $\rho_f$ and we give the following conjecture. 
\begin{con}
Assume that $f(x,y)$ has the property $P^{**}$. For sufficiently large $n$, the extremal graph with the largest $\rho_f$: 
\begin{itemize}
\item[(1)] in $\mathcal{T}_n$ is double star $DS_{\left\lceil\frac n2\right\rceil, \left\lfloor\frac n2\right\rfloor}$; 
\item[(2)] in $\mathcal{U}_n$ is $C_3\left(\left\lceil\frac{n-3}2\right\rceil, \left\lfloor\frac{n-3}2\right\rfloor, 0\right)$; 
\item[(3)] in $\mathcal{B}_n$ is $\Theta(1,2,2)\left(\left\lceil\frac{n-4}2\right\rceil, \left\lfloor\frac{n-4}2\right\rfloor\right)$, which is obtained from $\Theta(1,2,2)$ by adding $\left\lceil\frac{n-4}2\right\rceil, \left\lfloor\frac{n-4}2\right\rfloor$ pendant edges to two vertices of degree 3 of $\Theta(1,2,2)$ respectively. 
\end{itemize}
\end{con}

The conjectured extremal graphs all have an edge with a large edge weight in common: 
$$w_{max}=\begin{cases}
f\left(\left\lceil\frac n2\right\rceil, \left\lfloor\frac n2\right\rfloor\right) & \text{for } DS_{\left\lceil\frac n2\right\rceil, \left\lfloor\frac n2\right\rfloor}; \\ 
f\left(\left\lceil\frac{n+1}2\right\rceil, \left\lfloor\frac{n+1}2\right\rfloor\right) & \text{for } C_3\left(\left\lceil\frac{n-3}2\right\rceil, \left\lfloor\frac{n-3}2\right\rfloor, 0\right); \\
f\left(\left\lceil\frac{n+2}2\right\rceil, \left\lfloor\frac{n+2}2\right\rfloor\right) & \text{for } \Theta(1,2,2)\left(\left\lceil\frac{n-4}2\right\rceil, \left\lfloor\frac{n-4}2\right\rfloor\right). 
\end{cases}$$

\bibliographystyle{elsarticle-num}
\bibliography{Shen2023}

\end{document}